\documentclass[a4paper,reqno]{amsart}
\usepackage{amsfonts, color, latexsym, esint}
 
\theoremstyle{plain}
\begingroup
\newtheorem{theorem}{Theorem}[section]

\endgroup

\theoremstyle{definition}
\begingroup

\newtheorem{remark}[theorem]{Remark}

\endgroup

\theoremstyle{remark}

\numberwithin{equation}{section}

\newcommand{\Om}{\Omega}
\newcommand{\R}{\mathbb R}

\newcommand{\leb}{{\mathcal L}}
\newcommand{\mthree}{{\mathbb M}^{3{\times}3}}
\newcommand{\mtwo}{{\mathbb M}^{2{\times}2}}

\newcommand{\dist}{{\rm dist}}
\newcommand{\sym}{{\rm sym}\,}

\renewcommand{\div}{{\rm div}}
\newcommand{\wto}{\rightharpoonup}
\newcommand{\e}{\varepsilon}

\newcommand{\E}{{\mathcal E}}
\newcommand{\J}{{\mathcal J}}
\newcommand{\JvK}{{\mathcal J}_{\rm vK}}
\newcommand{\Jlin}{{\mathcal J}_{\rm lin}}
\newcommand{\Id}{{\rm Id}}
\newcommand{\nablah}{\nabla_{\!h}}

\title[The von K\'arm\'an plate equation as a limit of 3d elasticity]
{The time-dependent von K\'arm\'an plate equation as a limit of 3d nonlinear elasticity}
\author[H.~Abels]{Helmut Abels}
\author[M.G.~Mora]{Maria Giovanna Mora}
\author[S.~M\"uller]{Stefan M\"uller}
\address[H.~Abels]{NWF I -- Mathematik,
Universit\"at Regensburg,
93040 Regensburg, Germany} 
\email{helmut.abels@mathematik.uni-regensburg.de}
\address[M.G.~Mora]{Scuola Internazionale Superiore di Studi Avanzati, via Beirut 2, 34151 Trieste, Italy}
\email{mora@sissa.it}
\address[S.~M\"uller]{Hausdorff Center for Mathematics
\& Institute for Applied Mathematics, Universit\"at Bonn,
Endenicher Allee 60, 53115 Bonn, Germany}
\email{stefan.mueller@hcm.uni-bonn.de}

\begin{document}

\begin{abstract}
The asymptotic behaviour of the solutions of three-dimensional nonlinear elastodynamics in a thin plate
is studied, as the thickness $h$ of the plate tends to zero. Under appropriate scalings of the applied force
and of the initial values in terms of $h$, it is shown that three-dimensional solutions of the nonlinear elastodynamic 
equation converge to solutions of the time-dependent von K\'arm\'an plate equation.
\end{abstract}

\maketitle 

{\small

}

\bigskip
\bigskip

\section{Introduction}

This paper concerns the rigorous derivation of two-dimensional dynamic models for a thin elastic plate
starting from three-dimensional nonlinear elastodynamics.
To be definite, we consider a thin elastic plate of reference configuration $\Om_h:=\Om'{\times}(-\frac{h}{2},\frac{h}{2})$, 
where $\Om'\subset\R^2$ is a bounded domain with Lipschitz boundary and $h>0$.
We assume the plate to be made of a hyperelastic material 
whose energy potential $W\colon\mthree\to[0,+\infty]$ is a continuous function,
satisfying the following natural conditions:
\begin{align}
& W(RF) = W(F) \quad \text{for every } R\in SO(3), F\in\mthree \quad
\text{(frame indifference),} 
\label{h1}
\\
& W = 0 \quad \text{on } SO(3), 
\label{h2}
\\
& W(F)\geq C\, \dist^2(F, SO(3)), \quad C>0,
\label{h3}
\\
& 
W \text{ is } C^2 \text{ in a neighbourhood of } SO(3). \label{h4}
\end{align}
The dynamic equation of nonlinear elasticity reads as
\begin{equation}\label{dyneq-0}
\partial^2_{\tau} w-\div_{\!x} DW(\nabla w)= f^h
\quad \text{in }  [0,\tau_h]{\times}\Om_h, 
\end{equation}
where $w\colon [0,\tau_h]{\times}\Om_h\to\R^3$ is the deformation of the plate
and $f^h\colon [0,\tau_h]{\times}\Om_h\to\R^3$ is an external body force applied to the plate.
Equation \eqref{dyneq-0} is typically supplemented by the initial conditions
$$
w|_{\tau=0} = \bar w^h, \qquad
\partial_\tau w|_{\tau=0} = \hat w^h,
$$
and by boundary conditions, such as mixed Neumann--clamped boundary conditions:
\begin{equation}\label{BCs-0}
\begin{array}{c}
w|_{\partial\Om'{\times}(-\frac{h}{2},\frac{h}{2})}  =  x,
\medskip
\\
DW(\nabla w) e_3\big|_{x_3=\pm\frac{h}{2}}  =  0,
\end{array}
\end{equation}
or, assuming $\Om'=(-L,L)^2$, mixed Neumann--periodic boundary conditions:
\begin{equation}\label{PerBCs-0}
\begin{array}{c}
\big(w(\tau,x)-x\big)\big|_{x_\alpha=-L}
= \big(w(\tau,x)-x\big)\big|_{x_\alpha=L} \quad \alpha=1,2, 
\medskip
\\
DW(\nabla w) e_3 \big|_{x_3=\pm\frac{h}{2}}=0.
\end{array}
\end{equation}
The intent of this paper is to characterize the asymptotic behaviour of solutions to \eqref{dyneq-0},
as the thickness parameter $h$ tends to zero, by identifying the two-dimensional dynamic equation
satisfied by their limit as $h\to0$. Our purpose is to rigorously deduce a two-dimensional dynamic
model for a thin elastic plate.
Lower dimensional models for thin bodies are of great interest in elasticity theory, as they
are typically easier to handle both from an analytical and a numerical point of view than 
their three-dimensional counterparts. The problem of their rigorous derivation starting from 
the three-dimensional theory is in fact one of the main questions in elasticity. We refer to \cite{Ant, Cia, Lov} for a survey of the classical derivation approach 
and a discussion of the history of the subject.

\

Steady-state solutions of \eqref{dyneq-0} satisfy the stationary equation $-\div_{\!x} DW(\nabla w)= f^h$
in $\Om_h$, together with the boundary conditions \eqref{BCs-0} or \eqref{PerBCs-0}, which formally correspond
to the Euler-Lagrange equations of the energy functional
$$
\E^h(w)= \frac1h \int_{\Omega_h} W(\nabla w)\, dx -\frac1h \int_{\Omega_h} f^h{\,\cdot\,}w\, dx.
$$
It is therefore natural to look for local or global minimizers of $\E^h$.
The study of the asymptotic behaviour of global minimizers of $\E^h$, as $h\to0$,
can be performed through the analysis of the $\Gamma$-limit of $\E^h$ (see \cite{DM} for a
comprehensive introduction to $\Gamma$-convergence). 
To do this, it is convenient to rescale $\Om_h$ to a fixed domain $\Om:=\Om'{\times}(-\frac12,\frac12)$
and to rescale deformations according to this change of variables, by setting
$$ 
y(x) =\begin{pmatrix}y'(x) \smallskip\\ y_3(x)\end{pmatrix} := w(x', hx_3)
$$
for every $x=(x',x_3)\in\Om$. 
Assuming for simplicity that $f^h(x)=f^h(x')$, the energy functional $\E^h$ can be therefore written as
$$
\J^h(y):=\E^h(w)= \int_\Omega W(\nablah y)\, dx - \int_\Omega f^h{\,\cdot\,}y\, dx,
$$
where we have introduced the notation
$$
\nablah y:=\big(\nabla' y \,|\, \tfrac1h \partial_3 y \big).
$$
Let now $y^h$ be a minimizer of $\J^h$ subject, for instance, to the (rescaled) clamped boundary conditions
$$
y^h(x)=\begin{pmatrix}x' \\ hx_3\end{pmatrix} \quad \text{for } x\in\partial\Om'{\times}(-\tfrac{1}{2},\tfrac{1}{2}).
$$
The asymptotic behaviour of $y^h$, as $h\to0$, depends on the scaling of the applied force $f^h$
in terms of $h$. More precisely, if $f^h$ is of order $h^\alpha$ with $\alpha\geq0$, then
$\J^h(y^h)\leq Ch^\beta$, where $\beta=\alpha$ for $0\leq\alpha\leq2$ and $\beta=2\alpha-2$ for $\alpha>2$,
and $y^h$ converge in a suitable sense to a minimizer of the functional given by the $\Gamma$-limit of the sequence 
$h^{-\beta}\J^h$, as $h\to0$ (see \cite{FJM02, FJM06, LR95}). 
In particular, it has been shown in \cite{FJM06} that,
if $f^h$ is a normal force of the form $f^h(x')=h^\alpha f(x')e_3$ with $\alpha\geq3$ and $f\in L^2(\Om')$, 
then 
\begin{equation}\label{yhconv0}
y^h \ \to \ \begin{pmatrix}x' \\ 0\end{pmatrix} \quad \text{strongly in } H^1(\Omega;\R^3),
\end{equation}
that is, minimizers converge to the identity. This suggests to introduce the 
(scaled) in-plane and out-of-plane displacements defined by 
$$
u^h(x') := \frac{1}{h^{\alpha-1}}\int_{-\frac12}^{\frac12} 
\big((y^h)' -x' \big)\,dx_3, \qquad
v^h(x'):=\frac{1}{h^{\alpha-2}}\int_{-\frac12}^{\frac12}
y_3^h\,dx_3.
$$
As $h\to0$, $(u^h, v^h)$ converges strongly in $H^1$ to a limit displacement $(u,v)$,
which is a minimizer of the $\Gamma$-limit of $1/h^{2\alpha-2}\J^h$ (see \cite[Theorem~2]{FJM06}).
More precisely, if $\alpha=3$, then 
$(u,v)$ is a minimizer of the von K\'arm\'an plate functional
\begin{align*}
\JvK(u,v)=\frac12\int_{\Om'} Q_2\big(\sym&\nabla'u+\tfrac12 \nabla'v{\,\otimes\,}\nabla'v\big)\, dx'
\\ 
& {}+\frac{1}{24}\int_{\Om'}Q_2((\nabla')^2v)\, dx' -\int_{\Om'} fv\, dx' 
\end{align*}
with respect to the boundary conditions $u=0$, $v=0$, and $\nabla'v=0$ on $\partial\Om'$.
Here $Q_2:\mtwo\to\R$ is the quadratic form defined by
\begin{equation}\label{defQ2}
Q_2(G)=\leb_2 G {\,:\,}G :=\min_{F''=G} Q_3(F), 
\end{equation}
where $Q_3:\mthree\to\R$ is the quadratic form given by $Q_3(F):=D^2W(\Id)F{\,:\,}F$,
while $F''$ denotes the $2{\times}2$-submatrix of $F$ defined by $F''_{ij}=F_{ij}$
for $1\leq i,j\leq2$.

If instead $\alpha>3$, then the limit in-plane displacement $u$ is equal to $0$, while the
out-of-plane displacement $v$ is a minimizer of the linear plate functional
$$
\Jlin(v)=\frac{1}{24}\int_{\Om'}Q_2((\nabla')^2v)\, dx' -\int_{\Om'} fv\, dx'
$$
with respect to the boundary conditions $v=0$ and $\nabla'v=0$ on $\partial\Om'$.

This convergence result has been extended in \cite{Mue-Pak} to the case 
of a sequence of solutions of the equilibrium equation $-\div_{\!x} DW(\nabla w)= f^h$, 
assuming suitable growth conditions from above on the energy potential $W$.
This assumption has been removed in \cite{Mor-Sca}, but this requires to work with a 
different notion of stationarity, related to the Cauchy stress tensor balance law (see \cite{Ball}).
A different approach, based on centre manifold theory, was pursued by Mielke in \cite{Mie}
to compare solutions in a thin strip to a one-dimensional problem. Another related result 
is due to Monneau \cite{Mon}: given a sufficiently smooth and small solution of
the von K\'arm\'an equation, he proved the existence of a nearby three-dimensional solution.

\

In this paper we focus on the dynamical case with 
$f^h(\tau,x)=h^\alpha f(\tau,x')e_3$, $\alpha\geq3$, and $f\in L^2((0,+\infty);L^2(\Om'))$.
We also assume that the initial values $\bar w^h$, $\hat w^h$ have the following scaling in terms of $h$:
$$
\tfrac12 \int_{\Om_h}|\hat w^h(x)|^2\,dx + 
\int_{\Om_h} W(\nabla \bar w^h(x))\, dx\leq Ch^{2\alpha-1},
$$
which can be equivalently written on $\Omega$ as
\begin{equation}\label{hypind-0}
\tfrac12 \int_\Om |\hat w^h(x', hx_3)|^2\,dx + 
\int_\Om W(\nabla \bar w^h(x', hx_3))\, dx\leq Ch^{2\alpha-2}.
\end{equation}
Let $w^h$ be a solution to \eqref{dyneq-0} on $[0,\tau_h]{\times}\Omega_h$.
To discuss its limiting behaviour as $h\to 0$, it is convenient to rescale $\Om_h$ to the fixed
domain $\Om$, as before, and to rescale time by setting $t=h\tau$.
According to this change of variables, we set
$$ 
y^h(t,x):= w^h\big(\tfrac{t}{h},(x', hx_3)\big)
$$
for every $(t,x)\in (0,T_h){\times}\Om$, where $T_h:=h\tau_h$.
With this notation we have that the scaled deformations $y^h$ satisfy the equation
\begin{equation}\label{dyneq}
h^2\partial^2_{t} y^h-\div_h DW(\nabla_h y^h)= h^\alpha ge_3
\quad \text{in }  (0,T_h){\times}\Om, 
\end{equation}
where $g(t,x'):= f\big(\tfrac{t}{h},x'\big)$ for every $(t,x)\in (0,+\infty){\times}\Om'$ and 
the scaled divergence $\div_h \Phi$ of a given $\Phi\in H^1(\Om;\mthree)$ is
defined by
$$
\div_h\Phi{\,\cdot\,}e_i:=\sum_{j=1,2}\partial_j \Phi_{ij} +\frac1h \partial_3\Phi_{i3}
\qquad i=1,2,3.
$$
The scaled deformations $y^h$ satisfy the following initial conditions:
\begin{align}
& y^h(0,x)=\bar w^h(x',hx_3) \quad \text{for } x\in\Om,
\label{ICy}
\\
& \partial_t y^h(0,x)=\tfrac1h \hat w^h(x',hx_3) \quad \text{for } x\in\Om,
\label{ICdy}
\end{align}
together with the mixed Neumann--clamped boundary conditions
\begin{equation}\label{Dir}
\begin{array}{c}
y^h|_{\partial\Om'{\times}(-\frac12,\frac12)}=\begin{pmatrix}x' \\ hx_3\end{pmatrix},
\medskip
\\
DW(\nabla_h y^h) e_3\big|_{x_3=\pm\frac12}=0,
\end{array}
\end{equation}
or, respectively, assuming $\Om'=(-L,L)^2$, the mixed Neumann--periodic boundary conditions
\begin{equation}\label{PerDir}
\begin{array}{c}
\Big(y^h(t,x)-\begin{pmatrix}x' \\ hx_3\end{pmatrix}\Big)\Big|_{x_\alpha=-L}
= \Big(y^h(t,x)-\begin{pmatrix}x' \\ hx_3\end{pmatrix} \Big)\Big|_{x_\alpha=L} \quad \alpha=1,2, 
\medskip
\\
DW(\nabla_h y^h) e_3\big|_{x_3=\pm\frac12}=0.
\end{array}
\end{equation}
We note that \eqref{hypind-0} is equivalent to the following scaling condition on the initial
values of $y^h$:
$$
\tfrac12 h^2\int_\Om |\partial_t y^h(0,x)|^2\,dx + 
\int_\Om W(\nabla_h y^h(0,x))\, dx\leq Ch^{2\alpha-2}.
$$
The existence of a solution to \eqref{dyneq}, supplemented by the initial conditions 
\eqref{ICy}--\eqref{ICdy} and the mixed Neumann--periodic boundary conditions \eqref{PerDir}, 
is guaranteed by the recent results of~\cite{AMM09}. 
More precisely, we have proved in \cite[Theorem~3.1]{AMM09} that,  in the case $\alpha>3$,
under suitable regularity assumptions on $f$ and appropriate scaling and regularity of the 
initial data $\bar w^h$, $\hat w^h$ (compatible with \eqref{hypind-0}),
for every $T>0$ there exists $h_0>0$ such that a strong solution exists on $[0,T]$ 
for every $h\in(0,h_0)$.
In the case $\alpha=3$ we have shown that, if in addition $f$ is small enough on $[0,T]$,
a strong solution exists on $[0,T]$ for every $h\in(0,1)$.
In other words, we can assume that there exists a solution to \eqref{dyneq} on a time interval
$[0,T]$ independent of $h$.

In this paper we prove (Theorem~\ref{mainthm}) that, if $y^h$ is a weak solution to \eqref{dyneq} on $[0,T]$,
satisfying the initial conditions \eqref{ICy}--\eqref{ICdy}, the boundary conditions \eqref{Dir}
or \eqref{PerDir}, and the energy inequality, then convergence \eqref{yhconv0} still
holds uniformly in time. Moreover, the in-plane and out-of-plane displacements
$$
u^h(t,x') := \frac{1}{h^{\alpha-1}}\int_{-\frac12}^{\frac12} 
\big((y^h)' -x' \big)\,dx_3, \qquad
v^h(t,x'):=\frac{1}{h^{\alpha-2}}\int_{-\frac12}^{\frac12}y_3^h\,dx_3
$$
converge in a suitable sense to a limit displacement $(u,v)$. For $\alpha=3$ the limit
displacement $(u,v)$ is a solution to the dynamic von K\'arm\'an plate equations 
\begin{equation}\label{dynvK}
\begin{cases}
 \displaystyle \partial^2_{t} v +\tfrac{1}{12}\div\big[ \div \leb_2((\nabla')^2 v)\big]
-\div\big[\leb_2(\sym\nabla' u +\tfrac12 \nabla'v\otimes\nabla'v)\nabla'v\big]=g,
\smallskip
\\
\displaystyle
\div\big[\leb_2(\sym\nabla' u +\tfrac12 \nabla'v\otimes\nabla'v)\big]=0,
\end{cases}
\end{equation}
in $[0,T]{\times}\Om'$, and satisfies the boundary conditions
\begin{equation}\label{BCs}
u|_{\partial\Om'}=0, \quad  
v|_{\partial\Om'}=0, \quad  \nabla' v|_{\partial\Om'}=0,
\end{equation}
or, respectively, 
\begin{equation}\label{PerBCs}
u|_{x_\alpha=-L}=u|_{x_\alpha=L}, \quad  
v|_{x_\alpha=-L}=v|_{x_\alpha=L}, \quad  \nabla' v|_{x_\alpha=-L}= \nabla' v|_{x_\alpha=L},
\end{equation}
and the initial conditions
\begin{equation}\label{ICs}
v|_{t=0}=\bar w_3, \quad \partial_t v|_{t=0}=\hat w_3.
\end{equation}
Here $\leb_2$ is the linear form introduced in \eqref{defQ2},
while the limiting initial values $\bar w_3$ and $\hat w_3$ are the limits 
of suitably scaled averages of $\bar w^h_3$ and $\hat w^h_3$
(see \eqref{iv-conv}), which exist owing to the scaling condition \eqref{hypind-0}.

For $\alpha>3$ the limit in-plane displacement $u$ is equal to $0$, while the out-of-plane
displacement $v$ is a solution to the dynamic linear plate equation
\begin{equation}\label{dynlin}
\partial^2_{t} v +\tfrac{1}{12}\div\big[ \div \leb_2((\nabla')^2 v)\big]=g \qquad \text{in }
[0,T]{\times}\Om' 
\end{equation}
and satisfies the boundary conditions \eqref{BCs}, or, respectively, \eqref{PerBCs},
and the initial conditions \eqref{ICs}. 
This generalizes the convergence result of \cite[Theorem~4.1]{AMM09},
where we proved that for a special choice of the initial values and under the assumption
$W(F)=\dist^2(F, SO(3))$ the asymptotic development of the three-dimensional strong solutions of \eqref{dyneq}
can be characterized in the case $\alpha>3$ in terms of 
the solution $v$ of \eqref{dynlin}. 

To our knowledge, the present contribution, together with the results of \cite{AMM09},
is the first rigorous derivation of a 
lower dimensional elastodynamic model for a thin domain in the nonlinear framework. 
This problem has been extensively studied in the linear setting (see, e.g., \cite{Rao, Tam, Vod, X}), 
that is, performing the derivation starting from the three-dimensional linearized evolution model. 
However, since thin structures may undergo large rotations even under the action 
of very small forces, one cannot assume a priori the small strain condition, 
on which linearized elasticity is based. Our result implies, in particular, 
that the use of the two-dimensional dynamic linear plate equation \eqref{dynlin} is mathematically justified 
whenever the applied loads are of order $h^\alpha$ with $\alpha>3$
and the initial values satisfy \eqref{hypind-0}. 

We also mention a related result by Ge, Kruse, and Marsden \cite{GKM96}, where the problem of the limit 
of three-dimensional evolutionary elastic models to shell and rod models is addressed by
studying the convergence (in a suitable sense) of the underlying Hamiltonian structure. 
This approach however does not provide convergence of solutions.

\section{Statement and Proof of the Main Result}

This section is devoted to the proof of the following theorem, which is
the main result of the paper. We shall denote by $J_T$ the time interval given by
$[0,T]$ if $T\in(0,+\infty)$, and by $[0,+\infty)$ if $T=+\infty$.   

\begin{theorem}\label{mainthm}
Assume that \eqref{h1}--\,\eqref{h4} hold and that $W$ is differentiable and satisfies the growth condition
\begin{equation}\label{Lip}
|DW(F)|\leq C(|F|+1) \qquad \text{for every } F\in\mthree.
\end{equation}
Let $\alpha\geq3$ and let $(\hat w^h)\subset L^2(\Om_h;\R^3)$ and $(\bar w^h)\subset H^1(\Om_h;\R^3)$ be two
sequences satisfying
\begin{equation}\label{hypind}
\tfrac12 \int_\Om |\hat w^h(x', hx_3)|^2\,dx + 
\int_\Om W(\nabla \bar w^h(x',hx_3))\, dx\leq Ch^{2\alpha-2}.
\end{equation}
Let $T\in(0,+\infty]$, $g\in L^2((0,T);L^2(\Om'))$, and $h_0>0$. 
For every $h\in(0,h_0)$ let $y^h\in L^2((0,T);H^1(\Om;\R^3))$ with 
$$
\begin{array}{c}
\partial_t y^h \in L^2((0,T);L^2(\Om;\R^3)),
\smallskip
\\
\partial^2_{t} y^h \in L^2((0,T);H^{-1}(\Om;\R^3))
\end{array}
$$
be a weak solution to \eqref{dyneq} in $(0,T)$, satisfying the boundary conditions \eqref{Dir} (or,
assuming $\Om'=(-L,L)^2$, \eqref{PerDir}),
the initial conditions \eqref{ICy}--\,\eqref{ICdy}, and the energy inequality
\begin{align}
\tfrac12  h^2  \int_\Om &  |\partial_t y^h(t,x)|^2\, dx
+ \int_\Om W(\nablah y^h(t,x))\, dx
\nonumber
\\
 & \leq 
\tfrac12 \int_\Om |\hat w^h(x',hx_3)|^2\, dx
+ \int_\Om W(\nabla \bar w^h(x',hx_3))\, dx
\label{energy0}
\\
& \qquad 
{}+ \int_0^t \!\! \int_\Om h^\alpha g(s,x')\partial_t y_3^h(s,x)\, dxds
\nonumber
\end{align}
for a.e.\ $t\in(0,T)$. Then
\begin{equation}\label{yhconv}
y^h \to \begin{pmatrix} x' \\ 0  
\end{pmatrix} 
 \quad \text{strongly in } L^\infty_{loc}(J_T; H^1(\Om;\R^3)).
\end{equation}
Moreover, setting
$$
u^h(t,x') := \frac{1}{h^{\alpha-1}}\int_{-\frac12}^{\frac12} 
\big((y^h)' -x' \big)\,dx_3, \qquad
v^h(t,x'):=\frac{1}{h^{\alpha-2}}\int_{-\frac12}^{\frac12}y_3^h\,dx_3,
$$
the following assertions hold.
\medskip

\begin{itemize}
	\item[(i)] (von K\'arm\'an regime) Assume $\alpha=3$. 
Then, there exist $u\in L^\infty_{loc}(J_T;\allowbreak H^1(\Om';\R^2))$ and
$v\in L^\infty_{loc}(J_T; H^2(\Om'))\cap W^{1,\infty}_{loc}(J_T;\allowbreak L^2(\Om'))$,
with $\partial_t v\in C(J_T;H^{-3}(\Om'))$, 
such that, up to subsequences,
\begin{equation}\label{uconv}
u^h\wto u \quad \text{weakly}^* \text{ in } L^\infty_{loc}(J_T; H^1(\Om';\R^2))
\end{equation}
and
\begin{align}\label{vconv}
v^h\to v \quad \text{strongly in } L^\infty_{loc}(J_T; L^2(\Om')),
\\
\partial_t v^h \wto \partial_t v \quad \text{weakly}^* \text{ in } L^\infty_{loc}(J_T; L^2(\Om')),
\label{vconv2}
\end{align}
as $h\to 0$. The limit displacement $(u,v)$ is a weak solution in $(0,T)$ 
of the dynamic von K\'arm\'an plate equations \eqref{dynvK}, supplemented by
the boundary conditions \eqref{BCs} (or, respectively, \eqref{PerBCs}) and
the initial conditions \eqref{ICs}, where 
\begin{equation}\label{iv-conv}
\begin{array}{rl}
\displaystyle
\tfrac{1}{h}\int_{-\frac12}^{\frac12}\bar w_3^h(\cdot,hx_3)\,dx_3 \ \to \ \bar w_3
& \text{strongly in } H^1(\Om'),
\smallskip
\\
\displaystyle
\tfrac{1}{h^2}\int_{-\frac12}^{\frac12}\hat w_3^h(\cdot,hx_3)\,dx_3 \ \wto \ \hat w_3
& \text{weakly in } L^2(\Om').
\end{array}
\end{equation}
\medskip

	\item[(ii)] (linear regime) Assume $\alpha>3$. Then, \eqref{uconv} holds with $u=0$
and there exists 
$v\in L^\infty_{loc}(J_T; H^2(\Om'))\cap W^{1,\infty}_{loc}(J_T;\allowbreak L^2(\Om'))$,
with $\partial_t v\in C(J_T;H^{-3}(\Om'))$,  
such that, up to subsequences, \eqref{vconv}--\,\eqref{vconv2} hold. 
The limit displacement $v$ is a weak solution in $(0,T)$ to the dynamic linear plate equation
\eqref{dynlin}, supplemented by the boundary conditions \eqref{BCs} (or, respectively, \eqref{PerBCs}) and
the initial conditions \eqref{ICs}, where now
\begin{equation}\label{iv-conv-gen}
\begin{array}{rl}
\displaystyle
\frac{1}{h^{\alpha-2}}\int_{-\frac12}^{\frac12}\bar w_3^h(\cdot,hx_3)\,dx_3 \ \to \ \bar w_3
& \text{strongly in } H^1(\Om'),
\smallskip
\\
\displaystyle
\frac{1}{h^{\alpha-1}}\int_{-\frac12}^{\frac12}\hat w_3^h(\cdot,hx_3)\,dx_3 \ \wto \ \hat w_3
& \text{weakly in } L^2(\Om').
\end{array}
\end{equation}
\end{itemize}
\end{theorem}

\begin{remark}
The existence of the limits in \eqref{iv-conv} and \eqref{iv-conv-gen} 
is guaranteed by the scaling condition \eqref{hypind} (see Step~7 in the proof
of Theorem~\ref{mainthm}).
\end{remark}

\begin{remark}\label{rmk:weak}
We shall consider the following notion of weak solutions.
We say that a function $y^h\in L^2((0,T);H^1(\Om;\R^3))\cap H^1((0,T);L^2(\Om;\R^3))$
is a weak solution to \eqref{dyneq} in $(0,T)$ satisfying the boundary conditions \eqref{Dir} 
if $y^h=(x',hx_3)$ on $(0,T){\times}\partial\Om'{\times}(-\frac12,\frac12)$ and
the following equation is fulfilled:
$$
\int_0^T \!\!\!\int_\Om h^2\partial_t y^h{\,\cdot\,} \partial_t\varphi\, dxdt
- \int_0^T \!\!\!\int_\Om DW(\nablah y^h){\,:\,}\nablah\varphi\, dxdt
+ \int_0^T \!\!\!\int_\Om h^3 g\varphi_3\, dxdt =0 
$$
for every $\varphi\in L^2((0,T);H^1(\Om;\R^3))\cap H^1_0((0,T);L^2(\Om;\R^3))$ 
such that $\varphi=0$ on $(0,T){\times}\partial\Om'{\times}(-\frac12,\frac12)$.

Analogously, we say that a pair $(u,v)$ with $u\in L^\infty_{loc}(J_T; H^1(\Om';\R^2))$ and
$v\in L^\infty_{loc}(J_T; H^2(\Om'))\cap W^{1,\infty}_{loc}(J_T; L^2(\Om'))$ is a weak solution 
to \eqref{dynvK} in $(0,T)$, supplemented by the boundary conditions \eqref{BCs}, if 
\eqref{BCs} is satisfied and for every $T'\in J_T$ the following two equations
are fulfilled:
$$
\begin{array}{c}
\displaystyle
\int_0^{T'} \!\!\!\int_{\Om'} \partial_t v\partial_t\phi\, dx'dt
-\int_0^{T'} \!\!\!\int_{\Om'} \leb_2(\sym\nabla' u +\tfrac12 \nabla'v\otimes\nabla'v){\,:\,}
\nabla'v\otimes\nabla'\phi\, dx'dt
\medskip
\\
\displaystyle
{}-\int_0^{T'} \!\!\! \int_{\Om'} \tfrac{1}{12}\leb_2((\nabla')^2v){\,:\,}(\nabla')^2\phi
\, dx'dt
+ \int_0^{T'} \!\!\!\int_{\Om'} g\phi\, dx'dt=0
\end{array}
$$
for every $\phi\in L^2((0,T'); H^2_0(\Om'))\cap H^1_0((0,T'); L^2(\Om'))$, and
$$
\int_0^{T'} \!\!\! \int_{\Om'} \leb_2(\sym\nabla' u +\tfrac12 \nabla'v\otimes\nabla'v)
{\,:\,}\nabla'\psi\, dx'dt=0 
$$
for every $\psi\in L^2((0,T'); H^1_0(\Om';\R^2))$.
Finally, a function $v\in L^\infty_{loc}(J_T; H^2(\Om'))\cap W^{1,\infty}_{loc}(J_T; L^2(\Om'))$ 
is a weak solution to \eqref{dynlin} in $(0,T)$, supplemented by the boundary conditions \eqref{BCs}, if 
\eqref{BCs} is satisfied and for every $T'\in J_T$ the following equation
is fulfilled:
$$
\int_0^{T'} \!\!\!\int_{\Om'} \partial_t v\partial_t\phi\, dx'dt
-\int_0^{T'} \!\!\! \int_{\Om'} \tfrac{1}{12}\leb_2((\nabla')^2v){\,:\,}(\nabla')^2\phi
\, dx'dt
+ \int_0^{T'} \!\!\!\int_{\Om'} g\phi\, dx'dt=0
$$
for every $\phi\in L^2((0,T'); H^2_0(\Om'))\cap H^1_0((0,T'); L^2(\Om'))$.
\end{remark}

\begin{remark}
The notation $L^p_{loc}(J_T;X)$ denotes the space of all strongly measurable functions 
which are $p$-integrable (or essentially bounded if $p=\infty$) on every compact interval of $J_T$ 
with values in the Banach space $X$.
In particular, if $T\in(0,+\infty)$, the space $L^p_{loc}(J_T;X)$ coincides with $L^p((0,T);X)$, while,
if $T=+\infty$, $L^p_{loc}(J_T;X)$ is the space of functions belonging to $L^p((0,T');X)$ for every $T'<+\infty$. 
\end{remark}

Two of the main difficulties in the proof of Theorem~\ref{mainthm} are to show that the
deformation gradients must be close to the identity, because of the smallness of the applied force 
and of the initial data, and to derive enough compactness to pass to the limit
in the three-dimensional equation. The key remark is that the energy inequality \eqref{energy0}
satisfied by the solutions $y^h$, together with the scaling assumptions on the applied force
and the initial values, imply a corresponding precise scaling of the elastic part of the energy
(see \eqref{el-energy} below, in the case $\alpha=3$). By applying the quantitative rigidity
estimate proved in \cite[Theorem~3.1]{FJM02}, we can deduce from this bound on the elastic energy of $y^h$,
a decomposition of the deformation gradients $\nablah y^h$ into a rotation $R^h$ (depending only
on the in-plane variables) and a strain $G^h$ of order $h^{\alpha-1}$ (see \eqref{decomp} below,
in the case $\alpha=3$). The good controls on $R^h$ and $G^h$ provided by the rigidity estimate are
now the crucial ingredient to obtain the compactness properties needed to pass to the limit
in the three-dimensional equation. In particular, the following compactness criterion in the space 
$L^p((0,T);B)$, $B$ a Banach space, will be used. 

\begin{theorem}[{\cite[Theorem~6]{Simon}}]\label{thmsimon}
Let $X\hookrightarrow B\hookrightarrow Y$ be Banach spaces with compact imbedding $X\hookrightarrow B$.
Let $T\in(0,+\infty)$ and let $\mathcal F$ be a bounded subset of $L^\infty((0,T);X)$.
Assume that for every $0<t_1<t_2<T$
$$
\sup_{f\in{\mathcal F}}\|{\mathcal T}_s f -f\|_{L^1((t_1,t_2);Y)} \to 0, \quad \text{as }
s\to 0,
$$
where ${\mathcal T}_s f(t,x):=f(t+s,x)$ for every $t\in(-s,T-s)$ and $x\in X$.
Then $\mathcal F$ is relatively compact in $L^p((0,T);B)$ for every $1\leq p<\infty$.
\end{theorem}

We are now in a position to prove Theorem~\ref{mainthm}.
We prove the statement only in the case of the mixed Neumann--clamped boundary conditions \eqref{Dir} and for 
the scaling $\alpha=3$. The proof in the case of the mixed Neumann--periodic boundary conditions \eqref{PerDir}
or for the scaling $\alpha>3$ is completely analogous.

\begin{proof}[Proof of Theorem~\ref{mainthm}]
Let $\alpha=3$ and let $y^h$ be a weak solution to \eqref{dyneq} in $(0,T)$, satisfying
the mixed Neumann--clamped boundary conditions \eqref{Dir}, the initial conditions \eqref{ICy}--\eqref{ICdy},
and the energy inequality \eqref{energy0}.
The assumption \eqref{hypind} on the initial data and \eqref{energy0} imply that 
\begin{align}
\tfrac12 h^2 & \int_\Om |\partial_t y^h(t,x)|^2\, dx
+ \int_\Om W(\nablah y^h(t,x))\, dx
\nonumber
\\
& \leq 
Ch^4 + h^3 \Big(\int_0^t \!\! \int_{\Om'} |g(s,x')|^2\, dx'ds\Big)^{\!\frac12}\,
\Big(\int_0^t \!\! \int_\Om |\partial_t y_3^h(s,x)|^2\, dxds\Big)^{\!\frac12}
\label{energy-1}
\end{align}
for every $h\in(0,h_0)$ and a.e.\ $t\in(0,T)$.
By the Cauchy inequality we deduce that for every $T'\in J_T$
there exists a constant $C(T')>0$ such that
$$
\int_0^{T'} \!\! \int_\Om |\partial_t y^h|^2\, dxdt \leq C(T') h^2
$$
for every $h\in(0,h_0)$. Therefore, by \eqref{energy-1} we have that 
\begin{align}
& \displaystyle
\sup_{t\in[0,T']}\,\int_\Om |\partial_t y^h(t,x)|^2\, dx \leq C(T') h^2,
\label{kin-energy}
\\
& \displaystyle
\sup_{t\in[0,T']}\,\int_\Om W(\nablah y^h(t,x))\, dx \leq C(T') h^4.
\label{el-energy}
\end{align}
\medskip

\noindent{\bf Step 1.\ Construction of approximating rotations.}
By the energy estimate \eqref{el-energy} and by \cite[Theorem~6 and Remark~5]{FJM06} 
we can construct an approximating sequence $(R^h)$ 
in $L^\infty_{loc}(J_T; H^2(\Om';\mthree))$ 
such that $R^h(t,x')\in SO(3)$ for a.e.\ $(t,x')\in(0,T){\times}\Om'$ and
\begin{align}
& \displaystyle
\sup_{t\in[0,T']} \| \nablah y^h(t,\cdot) - R^h(t,\cdot) \|_{L^2(\Om)}\leq C(T') h^2,
\label{rig1}
\\
& \displaystyle
\sup_{t\in[0,T']} \| \nabla' R^h(t,\cdot)\|_{L^2(\Om')} 
+\sup_{t\in[0,T']} h\, \| (\nabla')^2 R^h(t,\cdot) \|_{L^2(\Om')}\leq C(T') h,
\label{rig2}
\\
& \displaystyle
\sup_{t\in[0,T']} \|R^h(t,\cdot)-\Id \|_{H^1(\Om')}\leq C(T') h
\label{rig3}
\end{align}
for every $T'\in J_T$.
By estimates \eqref{rig1} and \eqref{rig3} we deduce that
$$
\nablah y^h\to \Id \quad \text{strongly in } L^\infty_{loc}(J_T; L^2(\Om;\mthree)),
$$
hence $\partial_3 y^h\to 0$ and 
$$
\nabla y^h\to {\rm diag}(1,1,0) \quad \text{strongly in } L^\infty_{loc}(J_T; L^2(\Om;\mthree)).
$$
Since $|y^h(t,x)-(x', 0)|\leq \frac12h$ on $\partial\Om'\times(-\frac12,\frac12)$
for a.e.\ $t\in(0,T)$, the previous convergence together with the Poincar\'e inequality
implies \eqref{yhconv}.
\medskip

\noindent{\bf Step 2.\ Convergence of the sequence $A^h:=(R^h-\Id)/h$.}
Let us now consider the sequence 
$$
A^h:=\frac{R^h-\Id}{h}.
$$
By \eqref{rig3} there exists $A\in L^\infty_{loc}(J_T;H^1(\Om';\mthree))$ such that, 
up to subsequences,
\begin{equation}\label{Aconv}
A^h\wto A \quad \text{weakly}^* \text{ in } L^\infty_{loc}(J_T; H^1(\Om';\mthree)).
\end{equation}
We also notice that 
\begin{equation}\label{symR}
\sym\frac{R^h-\Id}{h^2}=-\frac{(A^h)^TA^h}{2},
\end{equation}
hence $\sym(R^h-\Id)/h^2$ is bounded in $L^\infty_{loc}(J_T; L^p(\Om';\mthree))$
for every $p<\infty$. In particular,
\begin{equation}\label{symAconv}
\sym A^h\to 0 \quad \text{strongly in } L^\infty_{loc}(J_T; L^p(\Om';\mthree))
\end{equation}
and $A$ is skew-symmetric.

We now claim that $(A^h e_\alpha)$ is strongly compact in 
$L^q_{loc}(J_T; L^p(\Om';\R^3))$ for $\alpha=1,2$ and any $1\leq q<\infty$, $2\leq p<\infty$. 
As $(A^he_\alpha)$ is uniformly bounded in $L^\infty_{loc}(J_T;H^1(\Om';\R^3))$, 
by Theorem~\ref{thmsimon} it is enough to show that for every $0<t_1<t_2<T$ and
any $h_j\to0$
\begin{equation}\label{scpt}
\lim_{s\to 0}\ \sup_{j}
\int_{t_1}^{t_2}\|A^{h_j}(t+s,\cdot)e_\alpha - A^{h_j}(t,\cdot)e_\alpha \|_{H^{-1}(\Om')}\, dt =0.
\end{equation}
We first observe that for a.e.\ $t\in (t_1,t_2)$ and $|s|<\delta$
\begin{align*}
&\|A^h(t+s,\cdot)e_\alpha - A^h(t,\cdot) e_\alpha \|_{H^{-1}(\Om')}
\\
& \leq  \tfrac1h \|R^h(t+s,\cdot)e_\alpha - \partial_\alpha y^h(t+s,\cdot) \|_{H^{-1}(\Om)}
\\
& \qquad
{}+ \tfrac1h \|\partial_\alpha y^h(t+s,\cdot) -\partial_\alpha y^h(t,\cdot) \|_{H^{-1}(\Om)}
+ \tfrac1h \|R^h(t,\cdot)e_\alpha - \partial_\alpha y^h(t,\cdot) \|_{H^{-1}(\Om)}.
\end{align*}
Owing to \eqref{rig1} there exists a constant $C(t_2)>0$ such that
$$
\tfrac1h \|R^h(\tau,\cdot)e_\alpha -\partial_\alpha y^h(\tau,\cdot) \|_{H^{-1}(\Om)}\leq
\tfrac1h \|R^h(\tau,\cdot)e_\alpha -\partial_\alpha y^h(\tau,\cdot) \|_{L^2(\Om)}\leq C(t_2)h
$$
for a.e.\ $\tau\in (t_1,t_2+\delta)$.
Moreover, in the same time interval we have
$$
\begin{array}{c}
\tfrac1h \|\partial_\alpha y^h(t+s,\cdot) -\partial_\alpha y^h(t,\cdot)  \|_{H^{-1}(\Om)}
 \leq \tfrac1h \|y^h(t+s,\cdot) - y^h(t,\cdot)\|_{L^2(\Om)} 
\smallskip
\\
\displaystyle \leq  \tfrac1h \int_t^{t+s} \|\partial_t y^h(\tau,\cdot)\|_{L^2(\Om)}\,d\tau
\leq C(t_2) |s|,
\end{array}
$$
where the last inequality follows from \eqref{kin-energy}.
Combining together all the previous inequalities, we conclude that
\begin{equation}\label{intest}
\int_{t_1}^{t_2}\|A^h(t+s,\cdot)e_\alpha - A^h(t,\cdot)e_\alpha \|_{H^{-1}(\Om')}\, dt\leq 
C(t_2)(2h+|s|)(t_2-t_1).
\end{equation}
Now, let $(h_j)$ be any sequence converging to $0$ and let us fix $\e>0$.
Clearly the supremum over the finite set 
$\{h_j : h_j\geq \e \}$ tends to zero as $s\to 0$, since
$$
\int_{t_1}^{t_2}\|f(s+t,\cdot)- f(t,\cdot)\|_{H^{-1}(\Om';\R^3)}\, dt \to 0
$$
for any fixed $f\in L^\infty_{loc}(J_T;L^2(\Om';\R^3))$. On the other hand,
by \eqref{intest} the supremum over the remaining set $\{ h_j : h_j < \e \}$ satisfies 
$$
\limsup_{s\to 0} \, \sup_{h_j<\e}\int_{t_1}^{t_2}
\|A^{h_j}(t+s,\cdot)e_\alpha - A^{h_j}(t,\cdot)e_\alpha \|_{H^{-1}(\Om')} \, dt\leq
2C(t_2)\e(t_2-t_1).
$$
Since $\e$ is arbitrary, this establishes \eqref{scpt} and, in turn,
strong compactness of $A^he_\alpha$ in $L^q_{loc}(J_T; L^p(\Om';\R^3))$.

Using the strong compactness of $A^he_\alpha$ in the identity \eqref{symR}
and the fact that $A$ is skew-symmetric,
we obtain that for every $\alpha,\beta=1,2$
\begin{equation}\label{symRc}
\Big(\sym\frac{R^h-\Id}{h^2}\Big)_{\alpha\beta}=
-\tfrac12(A^he_\alpha{\,\cdot\,} A^he_\beta)
\ \to \ -\tfrac12 (Ae_\alpha{\,\cdot\,} Ae_\beta)=\tfrac12 (A^2)_{\alpha\beta}
\end{equation}
strongly in $L^q_{loc}(J_T; L^2(\Om'))$ for every $1\leq q<\infty$.
\medskip

\noindent{\bf Step 3.\ Convergence of the displacements.}
{}From \eqref{rig1} and \eqref{symR} it follows that the symmetric part of
$\nabla' u^h$ is bounded in $L^\infty_{loc}(J_T; L^2(\Om';\mtwo))$.
Since $u^h(t,x')~=~0$ for $(t,x')\in(0,T){\times}\partial\Om'$,
the Korn-Poincar\'e inequality implies that $u^h$ is bounded in 
$L^\infty_{loc}(J_T; H^1(\Om';\R^2))$. Therefore, there exists a function
$u\in L^\infty_{loc}(J_T; H^1(\Om';\R^2))$ such that, up to subsequences, \eqref{uconv}
is satisfied.
In particular, we have that for every $T'\in(0,T)$
$$
\int_0^{T'} u^h(t,\cdot)\, dt \ \wto \ \int_0^{T'} u(t,\cdot)\, dt 
\quad \text{weakly in } H^1(\Om';\R^2),
$$
hence
$$
\int_0^{T'} u^h(t,\cdot)\, dt \ \to \ \int_0^{T'} u(t,\cdot)\, dt 
\quad \text{strongly in } L^2(\partial\Om';\R^2).
$$
Since $u^h=0$ on $\partial\Om'$ for a.e.\ $t$, 
this implies that $\int_0^{T'} u(t,x')\, dt=0$
for a.e.\ $x'\in\partial\Om'$ and every $T'\in(0,T)$, which yields
$u(t,x')=0$ for a.e.\ $(t,x')\in(0,T){\times}\partial\Om'$.
Moreover, passing to the limit in the identity
$$
h\partial_2 u^h_1= \tfrac1h \int_{-\frac12}^{\frac12}(\partial_2 y^h_1
-R^h_{12})\, dx_3 +A^h_{12},
$$
and owing to \eqref{uconv}, \eqref{rig1}, and \eqref{Aconv}, we deduce that
\begin{equation}\label{A12}
A_{12}=0.
\end{equation}

Using \eqref{rig1}, \eqref{rig3}, and the boundary condition
\begin{equation}\label{vBC}
v^h(t,x')=0 \quad \text{for } (t,x')\in(0,T){\times}\partial\Om',
\end{equation}
it is easy to see that $v^h$ is bounded in $L^\infty_{loc}(J_T; H^1(\Om'))$. 
Therefore, there exists $v\in L^\infty_{loc}(J_T; H^1(\Om'))$ such that, up to subsequences,
\begin{equation}\label{vconv3}
v^h\wto v \quad \text{weakly}^* \text{ in } L^\infty_{loc}(J_T; H^1(\Om')).
\end{equation}
Arguing as above, we infer from \eqref{vBC} and \eqref{vconv3} that
$v(t,x')=0$ for a.e.\ $(t,x')\in(0,T){\times}\partial\Om'$.
Moreover, the energy estimate \eqref{kin-energy} implies that $\partial_t v^h$ is bounded in 
$L^\infty_{loc}(J_T; L^2(\Om'))$.
This guarantees (see \cite{Simon}) that 
$v\in W^{1,\infty}_{loc}(J_T; L^2(\Om'))$ and that the
convergence properties \eqref{vconv} and \eqref{vconv2} are satisfied. 
Furthermore, from \eqref{rig1} and \eqref{Aconv} it follows that for $\alpha=1,2$
\begin{equation}\label{vder}
\partial_\alpha v= A_{3\alpha}.
\end{equation}
Since $A\in L^\infty_{loc}(J_T; H^1(\Om';\mthree))$, we deduce that
$v\in L^\infty_{loc}(J_T; H^2(\Om'))$.
Combining together \eqref{A12} and \eqref{vder}, we conclude that
\begin{equation}\label{vder2}
A=-\begin{pmatrix} \nabla'v \\ 0 \end{pmatrix}\otimes e_3+e_3\otimes\begin{pmatrix} \nabla'v \\ 0 \end{pmatrix}.
\end{equation}
Arguing as in \cite[Corollary~1]{FJM06}, one can show that the first moment
of the displacement, defined by
\begin{equation}\label{fod}
\zeta^h(t,x'):=\int_{-\frac12}^{\frac12} x_3\bigg(y^h-
\begin{pmatrix} x' \\ hx_3 \end{pmatrix} \bigg)\, dx_3,
\end{equation}
satisfies
$$
\frac{1}{h^2}\zeta^h\wto \frac{1}{12} Ae_3=-\frac{1}{12}\begin{pmatrix} \nabla'v \\ 0 \end{pmatrix}
\quad \text{weakly}^* \text{ in } L^\infty_{loc}(J_T;H^1(\Om';\R^3)).
$$
As $\zeta^h(t,x')=0$ for a.e.\ $(t,x')\in(0,T){\times}\partial\Om'$, 
the previous convergence, together with the
compactness of the trace operator from $H^1(\Om';\R^3)$ into $L^2(\partial\Om';\R^3)$, yields
that $\nabla' v(t,x')=0$ for a.e.\ $(t,x')\in(0,T){\times}\partial\Om'$.
\medskip

\noindent{\bf Step 4.\ Decomposition of the deformation gradient in rotation and strain.}
We now make use of the approximating sequence of rotations $R^h$ to decompose the deformation
gradients as
\begin{equation}\label{decomp}
\nablah y^h= R^h(\Id+h^2G^h), 
\end{equation}
where the sequence $G^h$ is bounded in $L^\infty_{loc}(J_T; L^2(\Om;\mthree))$
by \eqref{rig1}. Thus, up to extracting a subsequence, we have that
\begin{equation}\label{Gconv}
G^h\wto G \quad \text{weakly}^* \text{ in } L^\infty_{loc}(J_T; L^2(\Om;\mthree)).
\end{equation}
Arguing as in \cite[Lemma~2]{FJM06}, we find that for $\beta=1,2$
$$
\begin{array}{c}
\displaystyle
R^h(t,x')\frac{
G^h(t,x',x_3+\ell)e_\beta - G^h(t,x',x_3)
e_\beta}{l} 
\smallskip
\\
\displaystyle
=\partial_\beta\Big( \tfrac1\ell
\int_0^\ell \frac{\tfrac1h \partial_3 y^h(t,x',x_3+\tilde\ell)}{h}\, d\tilde\ell
\Big).
\end{array}
$$
Since $R^h$ converges to $\Id$ boundedly in measure on $(0,T'){\times}\Om$ 
for every $T'\in J_T$,
we have by \eqref{Gconv} that the left-hand side of the previous expression 
converges to the difference quotient
$(G(t,x',x_3+\ell)e_\beta - G(t,x',x_3)e_\beta)/\ell$ weakly$^*$ in 
$L^\infty_{loc}(J_T; L^2(\Om;\R^3))$, 
while the right-hand side converges to
$\partial_\beta Ae_3$ weakly$^*$ in $L^\infty_{loc}(J_T; H^{-1}(\Om;\R^3))$
by \eqref{rig1} and \eqref{Aconv}. Thus, we conclude that
$$
\frac{
G(t,x',x_3+\ell)e_\beta - G(t,x',x_3)
e_\beta}{\ell} =\partial_\beta A(t,x')e_3,
$$
hence there exists some $\bar G\in L^\infty_{loc}(J_T; L^2(\Om';\mthree))$ such that
$$
G(t,x',x_3)e_\beta= \bar G(t,x')e_\beta+x_3\partial_\beta A(t,x')e_3.
$$
Taking into account \eqref{vder2}, we deduce that for $\alpha,\beta=1,2$
\begin{equation}\label{linG}
G_{\alpha\beta}(t,x',x_3)= \bar G_{\alpha\beta}(t,x')- x_3\partial^2_{\alpha\beta} v(t,x').
\end{equation}

Let $\bar G''$ be the $2{\times}2$-submatrix given by $\bar G_{\alpha\beta}'':=\bar G_{\alpha\beta}$
for $1\leq\alpha,\beta\leq2$.
In order to identify the symmetric part of $\bar G''$, we first observe that
$$
\int_{-\frac12}^\frac12 \sym (R^hG^h)\, dx_3=
\int_{-\frac12}^\frac12 \sym\frac{\nablah y^h-\Id}{h^2}\, dx_3 -
\int_{-\frac12}^\frac12 \sym\frac{R^h-\Id}{h^2}\, dx_3.
$$
Passing to the limit and using \eqref{uconv}, \eqref{symRc}, and
\eqref{vder2}, we deduce that
\begin{equation}\label{barG}
\sym\bar G''=\sym\nabla' u +\tfrac12 \nabla' v\otimes\nabla' v .
\end{equation}
\medskip

\noindent{\bf Step 5.\ Convergence of the stress.}
We can now derive the limit equations satisfied by $u$ and $v$.
To this aim we set
$$
E^h:=\tfrac{1}{h^2}DW(\Id+h^2G^h).
$$
Then, by frame-indifference we have
$$
DW(\nablah y^h)= R^hDW(\Id+h^2G^h) =h^2R^hE^h,
$$
so that equation \eqref{dyneq} can be written in the weak form as
\begin{equation}\label{dyneq2}
\int_0^T \!\!\! \int_\Om \partial_t y^h{\,\cdot\,} \partial_t\varphi\, dxdt
- \int_0^T \!\!\! \int_\Om R^hE^h{\,:\,}\nablah\varphi\, dxdt
+ \int_0^T \!\!\!\int_\Om h g\varphi_3\, dxdt =0 
\end{equation}
for every $\varphi\in L^2((0,T);H^1(\Om;\R^3))\cap H^1_0((0,T);L^2(\Om;\R^3))$ 
such that $\varphi=0$ on $(0,T){\times}\partial\Om'{\times}(-\frac12,\frac12)$ (see Remark~\ref{rmk:weak}).
We also note that arguing as in \cite[Proposition~2.3]{Mue-Pak}, one can show that
\begin{equation}\label{Econv}
E^h\wto E:=\leb G \quad \text{weakly}^* \text{ in } L^\infty_{loc}(J_T; L^2(\Om;\mthree)),
\end{equation}
where the linear map $\leb$ on matrix space is given by $\leb:=D^2W(\Id)$.

Let $T'\in J_T$ and let $\varphi$ be a test function such that 
$\varphi=0$ on $(T',T){\times}\Om$.
Multiplying \eqref{dyneq2} by $h$ and passing to the limit as $h\to0$, we obtain
\begin{equation}\label{E3=0}
\int_0^{T'} \!\!\!\int_\Om Ee_3{\,\cdot\,}\partial_3\varphi\, dxdt=0 
\end{equation}
for every $\varphi\in L^2((0,T');H^1(\Om;\R^3))\cap H^1_0((0,T');L^2(\Om;\R^3))$ 
such that $\varphi=0$ on $(0,T'){\times}\partial\Om'{\times}(-\frac12,\frac12)$ and every $T'\in J_T$.
Here we have used \eqref{kin-energy} and the fact that $R^hE^h$ converges
to $E$ weakly$^*$ in $L^\infty((0,T'); L^2(\Om;\R^3))$, 
since $R^h$ converges to $\Id$ boundedly in measure on $(0,T'){\times}\Om'$.
Equality \eqref{E3=0} yields $Ee_3=0$ a.e.\ in $(0,T){\times}\Om$
and, since $E$ is symmetric by \eqref{Econv},
\begin{equation}\label{cons1}
E=
\begin{pmatrix}
E_{11} & E_{12} & 0 \\
E_{12} & E_{22} & 0 \\
0 & 0 & 0
\end{pmatrix}.
\end{equation}
\medskip

\noindent{\bf Step 6.\ Derivation of the limit equations.}
Let us introduce the zeroth and first moments of $E^h$, defined by
$$
\bar E^h(t,x')=\int_{-\frac12}^{\frac12} E^h(t,x)\,dx_3, \qquad
\hat E^h(t,x')=\int_{-\frac12}^{\frac12} x_3 E^h(t,x)\,dx_3,
$$
and let us fix $T'\in J_T$. 

Let $\psi\in L^2((0,T');H^1_0(\Om';\R^2))\cap H^1_0((0,T');L^2(\Om';\R^2))$. Choosing
$\varphi=(\psi,0)$ as test function in \eqref{dyneq2}, we obtain
$$
\int_0^{T'} \!\!\! \int_\Om \ \sum_{\alpha=1}^2\partial_t y^h_\alpha\partial_t\psi_\alpha\, dxdt
- \int_0^{T'} \!\!\! \int_{\Om'} \, \sum_{\alpha,\beta=1}^2 (R^h\bar E^h)_{\alpha\beta}
\partial_\beta\psi_\alpha\, dx'dt=0. 
$$
Passing to the limit as $h\to 0$ and using \eqref{kin-energy}, we deduce
\begin{equation}\label{mom0}
\int_0^{T'} \!\!\!\int_{\Om'} \, \sum_{\alpha,\beta=1}^2 \bar E_{\alpha\beta}
\partial_\beta\psi_\alpha\, dx'dt=0 
\end{equation}
for every $\psi\in L^2((0,T');H^1_0(\Om';\R^2))\cap H^1_0((0,T');L^2(\Om';\R^2))$.
By approximation \eqref{mom0} holds for every $\psi\in L^2((0,T');H^1_0(\Om';\R^2))$.

Let now $\phi\in L^2((0,T');H^2_0(\Om'))\cap H^1_0((0,T');H^1_0(\Om'))$.
Considering $\varphi=(0,\phi)$ as test function in \eqref{dyneq2}, we have
\begin{equation}\label{passo1}
\int_0^{T'} \!\!\! \int_{\Om'} \partial_t v^h\partial_t\phi\, dx'dt
- \int_0^{T'} \!\!\! \int_{\Om'} \ \sum_{\alpha=1}^2 \tfrac1h(R^h\bar E^h)_{3\alpha}
\partial_\alpha\phi\, dx'dt
+ \int_0^{T'} \!\!\!\int_{\Om'} g\phi \, dx'dt =0.
\end{equation}
We notice that
\begin{equation}\label{passo2}
\tfrac1h(R^h\bar E^h)_{3\alpha}= (A^h\bar E^h)_{3\alpha} +\tfrac1h\bar E^h_{3\alpha}.
\end{equation}
The strong compactness of $A^he_\beta$ in $L^2((0,T');L^p(\Om';\R^3))$
for $2\leq p<\infty$ and identity \eqref{vder} ensure that for $\beta=1,2$
$$
\int_0^{T'} \!\!\! \int_{\Om'} A^h_{3\beta}\bar E^h_{\beta\alpha}\partial_\alpha\phi\, dx'dt
\quad\longrightarrow \quad 
\int_0^{T'} \!\!\! \int_{\Om'} \bar E_{\beta\alpha}\partial_\beta v\partial_\alpha\phi\, dx'dt,
$$
while property \eqref{symAconv} implies that $A^h_{33}\to 0$ strongly in
$L^\infty((0,T');L^p(\Om'))$ for every $p<\infty$, hence
$$
\int_0^{T'} \!\!\! \int_{\Om'} A^h_{33}\bar E^h_{3\alpha}\partial_\alpha\phi\, dx'dt
\quad\longrightarrow \quad 0.
$$
These two convergence results, together with \eqref{vconv2}, \eqref{passo1}, and
\eqref{passo2}, guarantee that
\begin{align}
\lim_{h\to 0} &\int_0^{T'} \!\!\! \int_{\Om'} \sum_{\alpha=1}^2 \tfrac1h\bar E^h_{3\alpha}
\partial_\alpha\phi\, dx'dt
\nonumber \\
& = \int_0^{T'} \!\!\! \int_{\Om'} \partial_t v\partial_t\phi\, dx'dt
\label{cons2}
 -\int_0^{T'} \!\!\!\int_{\Om'} \bar E''{\,:\,}
\nabla'v{\,\otimes\,}\nabla'\phi\, dx'dt
+\int_0^{T'} \!\!\!\int_{\Om'} g\phi \, dx'dt
\end{align}
for every $\phi\in L^2((0,T');H^2_0(\Om'))\cap H^1_0((0,T');H^1_0(\Om'))$.

In order to derive the equation satisfied by the first moment $\hat E$, let us consider
$\eta\in L^2((0,T');H^1_0(\Om';\R^2))\cap H^1_0((0,T');L^2(\Om';\R^2))$. Choosing
$\varphi(t,x)=(x_3\eta(t,x'),0)$ as test function in \eqref{dyneq2}, we obtain
\begin{equation}\label{step}
\begin{array}{c}
\displaystyle
\int_0^{T'} \!\!\! \int_\Om \ \sum_{\alpha=1}^2 x_3\partial_t y^h_\alpha\partial_t\eta_\alpha\, dxdt
- \int_0^{T'} \!\!\! \int_{\Om'} \, \sum_{\alpha,\beta=1}^2 (R^h\hat E^h)_{\alpha\beta}
\partial_\beta\eta_\alpha\, dx'dt
\medskip
\\
\displaystyle
{}- \int_0^{T'} \!\!\! \int_{\Om'} \ \sum_{\alpha=1}^2 \tfrac1h(R^h\bar E^h)_{\alpha 3}
\eta_\alpha\, dx'dt=0.
\end{array}
\end{equation}
We note that for $\alpha=1,2$ and $j=1,2,3$
\begin{align}
& \displaystyle \label{passo3}
\tfrac1h(R^h\bar E^h)_{\alpha 3}= \sum_{j=1}^3 A^h_{\alpha j}\bar E^h_{j3} 
+\tfrac1h\bar E^h_{\alpha 3},
\\
& \displaystyle
A^h_{\alpha j}\bar E^h_{j3}= 
- A^h_{j\alpha}\bar E^h_{j3} + 2 
(\sym A^h)_{j\alpha}\bar E^h_{j3}.
\end{align}
Using again \eqref{symAconv} and the strong compactness of $A^he_\beta$ in 
$L^2((0,T');L^p(\Om';\R^3))$ for $2\leq p<\infty$,
we deduce from the previous decomposition
$$
\int_0^{T'} \!\!\! \int_{\Om'} \ \sum_{j=1}^3 A^h_{\alpha j}\bar E^h_{j3}\eta_\alpha\, dx'dt
\quad\longrightarrow\quad 
-\int_0^{T'} \!\!\! \int_{\Om'} \ \sum_{j=1}^3 A_{j\alpha}\bar E_{j3}\eta_\alpha\, dx'dt=0,
$$
where the last equality follows from \eqref{vder2} and \eqref{cons1}. 
Therefore, passing to the limit in \eqref{step} and using the energy estimate
\eqref{kin-energy} and the decomposition \eqref{passo3}, we obtain
\begin{equation}\label{cons3}
\int_0^{T'} \!\!\! \int_{\Om'} \ \sum_{\alpha=1}^2 \tfrac1h\bar E^h_{\alpha 3}
\eta_\alpha\, dx'dt 
\quad\longrightarrow\quad  
- \int_0^{T'} \!\!\! \int_{\Om'} \, \sum_{\alpha,\beta=1}^2 \hat E_{\alpha\beta}
\partial_\beta\eta_\alpha\, dx'dt.
\end{equation}
Using the identity
$$
E^h-(E^h)^T=-h^2 \big( E^h(G^h)^T - G^h(E^h)^T
\big),
$$
one can prove that 
\begin{equation}\label{L1est}
\sup_{t\in[0,T']}\|E^h-(E^h)^T\|_{L^1(\Om)}\leq C(T')h^2 
\end{equation}
(see \cite[Step~4 in the proof of Theorem~1.1]{Mue-Pak}). Choosing $\eta=\nabla'\phi$
in \eqref{cons3} and combining it with the previous remark and \eqref{cons2}, we conclude that
\begin{equation}\label{fin}
\begin{array}{c}
\displaystyle
\int_0^{T'} \!\!\! \int_{\Om'} \partial_t v\partial_t\phi\, dx'dt
-\int_0^{T'} \!\!\! \int_{\Om'} \bar E''{\,:\,}\nabla'v\otimes\nabla'\phi\, dx'dt
\medskip
\\
\displaystyle
\qquad{}+ \int_0^{T'} \!\!\! \int_{\Om'} \hat E''{\,:\,}(\nabla')^2\phi\, dx'dt
+ \int_0^{T'} \!\!\! \int_{\Om'} g\phi\, dx'dt=0.
\end{array}
\end{equation}
for every $\phi\in L^2((0,T');H^2_0(\Om')\cap W^{1,\infty}(\Om'))\cap H^1_0((0,T');H^1_0(\Om'))$.
By approximation \eqref{fin} holds for every $\phi\in L^2((0,T');H^2_0(\Om'))\cap H^1_0((0,T');H^1_0(\Om'))$.

By \cite[Proposition~3.2]{Mue-Pak} and \eqref{linG} we obtain
$$
E''=\leb_2 G''=\leb_2 \bar G'' -x_3\leb_2((\nabla')^2v).
$$
As a consequence of this equality and of \eqref{barG}, we have 
$$
\bar E''=\leb_2 (\sym \bar G'')=\leb_2(\sym \nabla' u +\tfrac12 \nabla'v\otimes\nabla'v),
$$
while
$$
\hat E=-\tfrac{1}{12}\leb_2((\nabla')^2v).
$$
Owing to the last two identities and to equations \eqref{mom0} and \eqref{fin},
the limit displacement $(u,v)$ is a weak solution to \eqref{dynvK} (see Remark~\ref{rmk:weak}).
\medskip

\noindent{\bf Step 7.\ Derivation of the initial condition.}
It remains to prove that $v$ satisfies the initial conditions \eqref{ICs}.
First we observe that assumption \eqref{hypind} implies that, up to subsequences,
\begin{equation}\label{icto+}
\tfrac{1}{h^2}\int_{-\frac12}^{\frac12}\hat w_3^h(\cdot,hx_3)\,dx_3 \ \wto \ \hat w_3
\end{equation}
weakly in $L^2(\Om')$, and that 
\begin{equation}\label{icto}
\tfrac{1}{h}\int_{-\frac12}^{\frac12}\bar w_3^h(\cdot,hx_3)\,dx_3\ \to \ \bar w_3
\end{equation}
strongly in $H^1(\Om')$, owing to \cite[Lemma~13]{LM} (the proof of the Lemma can be easily
adapted to cover also the case $\alpha>3$).

Since $W^{1,\infty}((0,T'); L^2(\Om'))$ embeds into $C([0,T']; L^2(\Om'))$, we have
that $v^h, v\in C([0,T']; L^2(\Om'))$ for every $h\in(0,h_0)$, so that by \eqref{ICy}
and \eqref{vconv} 
$$
v^h(0,\cdot)=\tfrac{1}{h}\int_{-\frac12}^{\frac12}\bar w_3^h(\cdot,hx_3)\,dx_3
\ \to \ v(0,\cdot)
$$
strongly in $L^2(\Om')$.
By \eqref{icto} we conclude that $v(0,x')=\bar w_3(x')$ for a.e.\ $x'\in\Om'$.

Using the decompositions \eqref{passo2} and \eqref{passo3}, and the estimate \eqref{L1est},
we deduce from equations \eqref{passo1} and \eqref{step} that there exists a constant $C>0$,
independent of $h$, such that
$$
\Big|\int_0^{T'} \!\!\! \int_{\Om'} \Big(
\partial_t v^h\partial_t\phi
-\sum_{\alpha=1}^2 \partial_t \zeta^h_\alpha\partial_t\partial_\alpha\phi
\Big) \, dx'dt\Big|
\leq C\|\phi \|_{L^2((0,T');H^3_0(\Om'))}
$$
for every $\phi\in C^\infty_c((0,T'){\times}\Om')$. Here $\zeta^h$ is the first moment
of the displacement introduced in \eqref{fod}.
This implies that the sequence 
$$
\partial^2_{t} v^h+ \sum_{\alpha=1}^2 \partial^2_{t}\partial_\alpha 
\zeta^h_\alpha
$$ 
is uniformly bounded in $L^2((0,T');H^{-3}(\Om'))$.
On the other hand, by \eqref{kin-energy} the sequence $(\partial_t\partial_\alpha 
\zeta^h_\alpha)$ converges to $0$ strongly in $L^\infty((0,T');H^{-1}(\Om'))$; thus, 
by \eqref{vconv2} we conclude that
$$
\partial_t v^h+ \sum_{\alpha=1}^2 \partial_t\partial_\alpha 
\zeta^h_\alpha \wto \partial_t v \quad
\text{weakly}^* \text{ in } L^\infty((0,T');H^{-1}(\Om')).
$$
As $H^1((0,T'); H^{-3}(\Om'))\cap L^\infty((0,T');H^{-1}(\Om'))$ 
embeds compactly into the space $C([0,T']; H^{-3}(\Om'))$, it follows that
$$
\|\partial_t v^h(t,\cdot)+\sum_{\alpha=1}^2 \partial_t\partial_\alpha 
\zeta^h_\alpha(t,\cdot) - \partial_t v(t,\cdot)\|_{H^{-3}(\Om')}\ \to \ 0
$$
uniformly for $t\in[0,T']$. In particular, we have 
$$
\partial_t v^h(0,\cdot)+ \sum_{\alpha=1}^2 \partial_t\partial_\alpha 
\zeta^h_\alpha(0,\cdot)\ \to \ \partial_t v(0,\cdot)
$$
strongly in $H^{-3}(\Om')$. The initial condition \eqref{ICdy} and the estimate \eqref{hypind} guarantee that
$\partial_t\partial_\alpha \zeta^h_\alpha(0,\cdot)$ converge to $0$ strongly in $H^{-1}(\Om')$.
Therefore, by \eqref{ICdy} and \eqref{icto+} 
we deduce that $\partial_t v(0,x')=\hat w_3(x')$ for a.e.\ $x'\in\Om'$.
This concludes the proof.
\end{proof}

\bigskip
\medskip
\noindent
\textbf{Acknowledgments.}
This work was partially supported by GNAMPA, 
through the project ``Problemi di riduzione di dimensione per strutture elastiche sottili'' 2008.

\bigskip


\end{document}